\documentclass[12pt]{article}
\usepackage{amsfonts,amssymb,latexsym, amsmath, amsthm}

\usepackage[usenames]{color}

\newtheorem{thm}{Theorem}[section]

\theoremstyle{remark}

\theoremstyle{plain}

\numberwithin{equation}{section}

\def\leq{\leqslant}
\def\geq{\geqslant}
\def\Aut{{\rm Aut}}

\def\RR{{\mathbb R}}

\newcommand{\Om}{\Omega}

\def\Z{{\mathbb Z}}

\def\QQ{{\mathbb Q}}

\def\mC{{\mathcal C}}

\def\fK{{\mathfrak K}}

\def\C{{\mathrm C}}
\def\dep{\mathrm {dep}}

\newtheorem{defi}{Definition}[section]
\newtheorem{theorem}{\bf Theorem}

\newtheorem{lemma}[defi]{Lemma}
\newtheorem{prop}[defi]{Proposition}

\newtheorem{cor}[defi]{Corollary}

\newtheorem{remark}[defi]{Remark}

\newtheorem{*example}[defi]{*Example}

\addtocontents{ences.txt}{References}

\def\suppo{\mathrm{supp}}
\def\rst{\mathrm{rst}}
\def\St{\mathrm{st}}

\def\mro{\mathrm{o}}

\begin{document}

\title{Recognizing the real line} 
\author{A. M. W. Glass and John S. Wilson} \small{\date{\today}} \maketitle { \def\thefootnote{} \footnote{2010 AMS Classification: 20B07, 06F15, 03C60, 05C05.

 }

\setcounter{theorem}{0}
\setcounter{footnote}{0}

\begin{abstract}

Let $(\Omega, \leq)$ be a totally ordered set.  We prove that if $\Aut(\Omega,\leq)$ is transitive and satisfies the same first-order sentences as $\Aut(\RR,\leq)$ (in the language of groups) then $\Omega$ and $\RR$ are isomorphic ordered sets. This improvement of a theorem of Gurevich and Holland is obtained as a consequence of a study of centralizers associated with certain transitive subgroups of $\Aut(\Omega,\leq)$. 
\end{abstract}

\section{Introduction}  

In 1981, Gurevich and Holland \cite{GH} proved the following result.

\begin{thm} \label{GH}  Suppose that $(\Omega,\leq)$ is a totally ordered set such that $\Aut(\Omega,\leq)$ acts transitively on pairs $(\alpha,\beta)$ with $\alpha< \beta$. 
If $\Aut(\Omega,\leq)$ and $\Aut(\RR,\leq)$ satisfy the same first-order sentences,  then $\Omega$ is isomorphic to $\RR$ as an ordered set. \end{thm}

Other results of a similar kind were obtained in \cite{GGHJ}. Theorem \ref{GH} required first-order sentences in the language of $\ell$-groups, which is richer than the language of groups, but a slight extension shows that in fact only sentences in the language of groups are needed (see \cite[Theorem 2B*]{G81}):

\begin{cor} \label{GHgp}  Suppose that $(\Omega,\leq)$ is a totally ordered set such that $\Aut(\Omega,\leq)$ acts transitively on pairs $(\alpha,\beta)$ with $\alpha< \beta$. 
If $\Aut(\Omega,\leq)$ and $\Aut(\RR,\leq)$ satisfy the same first-order sentences in the language of groups, then $\Omega$ is isomorphic to $\RR$ as an ordered set. \end{cor}

Here we establish the following improvement to Theorem \ref{GH}. 

\begin{theorem} \label{reals} Suppose that $(\Omega,\leq)$ is a totally ordered set on which $\Aut(\Omega,\leq)$ acts transitively, 
and that $\Aut(\Omega,\leq)$ and $\Aut(\RR,\leq)$ satisfy the same first-order  sentences in the language of groups.  Then $\Omega$ is isomorphic to $\RR$ as an ordered set.
\end{theorem}

Transitivity is necessary in the above result.  Let $\Lambda$ be any rigid totally ordered set with at least two elements (for example a finite totally ordered set with at least two elements), and let $\Omega = \Lambda \times \RR$, with the order defined by $(\lambda_1,r_1)< 
(\lambda_2,r_2)$ if $r_1<r_2$ or if both $r_1=r_2$ and $\lambda_1<\lambda_2$.
It is easy to see that $\Aut(\Lambda \times \RR, \leq)$ is isomorphic to $\Aut(\RR,\leq)$.

Similar arguments allow us to strengthen other known results.   

\begin{theorem} \label{GHQ2}  Suppose that $(\Omega,\leq)$ is a totally ordered set on which $\Aut(\Omega,\leq)$ acts transitively. If $\Aut(\Omega,\leq)$ and $\Aut(\QQ,\leq)$ satisfy the same first-order sentences in the language of  groups then $\Omega$ is isomorphic as an ordered set to $\QQ$ or $\RR \setminus \QQ$. \end{theorem} 

The corresponding result with the stronger hypothesis that
$\Aut(\Omega,\leq)$ acts o-$2$ transitively on $\Om$, i.e., transitively on pairs $(\alpha,\beta)$ with $\alpha< \beta$, is a slight extension of a result of
Gurevich and Holland \cite{GH} (cf.\ \cite[Theorem 2C*]{G81}). 

The breakthrough in this work comes from employing a technique in \cite{W}: we use double centralizers of certain subsets of groups of order-preserving automorphisms of totally ordered sets $\Omega$ to give first-order expressibility of certain convex subsets of $\Omega$.  The ideas have implications for a large family of subgroups of the groups $\Aut(\Omega,\leq)$ (see \cite{arx}).

\section{Preliminaries and a reduction}

We write $X\subset Y$ for $X\subseteq Y$ and $X\neq Y$.  Our notation for conjugates and commutators is in accordance with our use of right actions: we write $g^f$ for $f^{-1}gf$ and $[f,g]$ for $f^{-1}g^{-1}fg$.

Automorphism groups $\Aut(\Omega,\leq)$ of totally ordered sets $(\Omega,\leq)$ are closed under taking the pointwise maximum $f\vee g$ and pointwise minimum $f\wedge g$
 of elements $f,g$ defined, respectively, by $$\alpha(f\vee g)= \max\{\alpha f,\alpha g\}\quad\hbox{and}\quad 
\alpha(f\wedge g)= \min\{\alpha f,\alpha g\}\quad \hbox{for all }\alpha\in \Omega.$$
An $\ell$-{\em permutation group} $(G,\Omega)$ is a subgroup of $\Aut(\Omega,\leq)$ closed under the binary operations $\vee$ and $\wedge$. 
Transitive $\ell$-permutation groups are of particular interest, and all groups studied in this paper will be assumed to be transitive.

Let $(G,\Omega)$ be a transitive $\ell$-permutation group.  A $G$-congruence on the set $\Omega$ is an equivalence relation  ${\mathcal C}$ on $\Om$ such that $(\alpha g)\, {\mathcal C}\,(\beta g)$ whenever $\alpha\, {\mathcal C}\,\beta$ and $g\in G$ ($\alpha,\beta\in \Om$). A
{\em convex $G$-congruence} $\mathcal C$ on $\Omega$ is a $G$-congruence  with all  ${\mathcal C}$-classes convex; these classes are called {\em o-blocks}.  
We suppress the mention of $G$ if it is clear from context.
By transitivity, each o-block $\Delta$ is a class of a unique convex
congruence; its set of classes is $\{ \Delta g\mid g\in G\}$.
We denote this convex congruence by $\kappa(\Delta)$.
  
\begin{prop} \label{ccc} {\rm (\cite[Theorem 3.A]{G81}) \em
The set of convex congruences of a transitive $\ell$-permutation group is totally ordered by inclusion. }
\end{prop}

If ${\mathcal C}$ and ${\mathcal D}$ are convex congruences with ${\mathcal C}\subset {\mathcal D}$ and there is no convex congruence strictly between ${\mathcal C}$ and ${\mathcal D}$, then we say that 
$\mathcal D$ {\em covers} $\mathcal C$ and
${\mathcal C}$ is {\em covered} by ${\mathcal D}$. 
Let $\alpha,\beta\in \Omega$ be distinct.  Then both the union $U(\alpha,\beta)$ of all convex congruences $\mC$ for which $\alpha$, $\beta$ lie in distinct o-blocks and the intersection $V(\alpha,\beta)$ of all convex congruences $\mC$ for which $\alpha$, $\beta$ lie in the same o-block are convex congruences. Clearly, $U(\alpha,\beta)$ is covered by $V(\alpha,\beta)$.  Let 
$$\fK=\{V(\alpha,\beta)\mid\alpha,\beta\in\Omega,\alpha\neq\beta\}.$$
Thus $\fK$ is totally ordered by inclusion. 
It is called the {\em spine} of $(G,\Omega)$.  
For all $\alpha,\beta\in \Omega$ we have $\beta=\alpha g$ for some $g\in G$  by transitivity. 
Therefore $\fK$ can also be described as follows:
$$\fK=\{V(\alpha,\alpha g)\mid\alpha\in\Omega,g\in G, \alpha g\neq \alpha\}. $$
Write $T$ for the set of o-blocks of elements of $\fK$. If $\Delta\in
T$, then $\kappa(\Delta)\in \fK$ and so $\kappa$ restricts to a surjective map from $T$ to $\fK$.
For each $\mC \in \fK$, write $\pi(\mC)$ for both the convex congruence covered by $\mC$ and its set of  o-blocks; 
the latter inherits a total order from $\Omega$.
If  $\Delta$ is a $\mC$-class, let $\pi(\Delta)$ be the set of all $\pi(\mC)$-classes contained in $\Delta$.

We define the {\em stabilizer} $\St(\Delta)$  and {\em rigid stabilizer} $\rst(\Delta)$ of  an o-block $\Delta$ 
as follows:
$$\St(\Delta):=\{g\in G \mid \Delta g=\Delta\}\quad\hbox{ and }\quad\rst(\Delta):=  \{ g\in G\mid \suppo(g)\subseteq \Delta\},$$
where $\suppo(g):=\{ \alpha\in \Omega\mid \alpha g\neq \alpha\}.$ 
So $\St(\Delta)$ and $\rst(\Delta)$ are convex sublattice subgroups of $G$ and $\rst(\Delta)\subseteq\St(\Delta)$.

Each $g\in \St(\Delta)$ induces an automorphism $g_\Delta$ of the ordered set $\pi(\Delta)$ given by $$\Gamma  g_\Delta=\Gamma g\qquad \hbox{for all } \ \Gamma\in \pi(\Delta).$$   Let 
$$ G(\Delta):=\{ g_\Delta\mid g\in \St(\Delta)\}.$$

 Note that $(G(\Delta),\pi(\Delta))$ is transitive and o-primitive.  Furthermore, if $K\in \fK$ and $\Delta, \Delta'$ are both $K$-classes, then $(G(\Delta),\pi(\Delta))$ and 
$(G(\Delta'),\pi(\Delta'))$ are isomorphic, an isomorphism being induced by conjugation by any $f\in G$ with $\Delta f=\Delta'$  since $(\Gamma f)(f^{-1}gf)=(\Gamma g)f$ for all $g\in \rst(\Delta),\;\Gamma\in \pi(\Delta)$.  
It is customary to write $(G_K,\Omega_K)$ for any of these $\ell$-permutation groups; they are independent of the o-block $\Delta$ of $K$ to within $\ell$-permutation isomorphism. 
 
For each $g \in G$ and each subset $\Lambda$ of $\Om$ that is a union of 
convex $g$-invariant subsets  of $\Om$, write
$\dep(g,\Lambda)$ for the element of $\Aut(\Omega,\leq)$ that
agrees with
$g$ on $\Lambda$ and with the identity elsewhere.
We say that $(G,\Om)$ is {\em fully depressible} if
$\dep(g,\Lambda)\in
G$ for all $g\in G$ and all such $\Lambda\subseteq \Om$. 
In particular, if $(G,\Om)$ is fully depressible, $\Delta\in T$ and
$g\in \St(\Delta)$, then $\dep(g,\Lambda)\in \rst(\Delta)\subseteq G$.
Moreover, the action of $\{ g_\Delta\mid g\in \rst(\Delta)\}$ on
$\pi(\Delta)$ is equal to $(G(\Delta),\pi(\Delta))$ in this case for
every $\Delta\in T$.  Clearly $\Aut(\Omega,\leq)$ itself is fully
depressible.
 
If $G$ is transitive on all $n$-tuples $(\alpha_1,\dots,\alpha_n)\in \Omega^n$ with $\alpha_1< \dots < \alpha_n$, we say that $(G,\Omega)$ is o-$n$ {\em transitive}. 
We shall need the following result (see \cite[Lemma 1.10.1]{G81}):

\begin{lemma} \label{2=n}
Every $\mro$-$2$ transitive $\ell$-permutation group $(G,\Om)$ is $\mro$-$n$ transitive for all integers $n\geq2$.
\end{lemma} 

We also need an immediate consequence of McCleary's Trichotomy \cite{Mc}:

\begin{prop}\label{dichot}
Let $(G,\Omega)$ be a transitive fully  depressible $\ell$-permutation group. Then $(G,\Om)$ is  $\mro$-primitive if and only if either
\begin{enumerate}
\item[\rm (I)] $(\Omega,\leq)$ is order-isomorphic to a subgroup of the reals and the action of $G$ on $\Om$ is the right regular representation$;$ or
\item[\rm  (II)] $(G,\Omega)$ is $\mro$-$2$ transitive$.$ 
\end{enumerate}
\end{prop}

Transitive $\mro$-primitive $\ell$-permutation groups of type (II) are non-abelian. 

For each $h\in G$, let  $$X_{h}:=\{ [h^{-1},h^g] \mid g\in G\} \quad \hbox{and}\quad
W_{h}=\bigcup\{X_{h^g}\mid g\in G,\; [X_{h},X_{h^g}]\neq 1\}.$$
The sets $X_h$, $W_h$ are evidently definable in the first-order language of group theory.

For any subset $S$ of $G$, we write $\C^2_G(S)$ as shorthand for $\C_G(\C_G(S))$, the double centralizer of $S$ in $G$. If $S$ is definable in $G$ in the first-order language of group theory  then so is $\C_G^2(S)$.

\begin{prop} \label{propoprim}
Let $(G,\Om)$ be a transitive fully  depressible $\ell$-permut\-ation group.
Then $(G,\Om)$ is $\mro$-primitive if and only if $\C^2_G(W_g)=G$ for all $g\in G\setminus \{ 1\}$.
\end{prop}

This result has the following immediate consequence.

\begin{cor} If $(G_1,\Omega_1)$, $(G_2,\Omega_2)$ are transitive fully depressible $l$-groups that satisfy the same first-order sentences in the language of group theory,
and one of these groups is $\mro$-primitive, then so is the other.  \end{cor}

We can now deduce Theorems \ref{reals} and \ref{GHQ2}. 
Let $\Lambda=\RR$ or $\Lambda=\QQ$.
Then $\Aut(\Lambda,\leq)$  acts o-$2$-transitively on $\Lambda$ and so it is o-primitive and non-abelian.  Thus if $(G,\Omega)$ is a transitive fully  depressible $\ell$-permutation group satisfying the same first-order sentences (in the language of groups) as ${\rm Aut}(\Lambda,\leq)$, then $(G,\Om)$ is non-abelian and $G$ acts o-primitively on $\Omega$ by Proposition \ref{propoprim}.  Hence $G$ acts o-$2$-transitively by Proposition \ref{dichot}.  Theorems A and B now follow directly from 
Corollary \ref{GHgp} and the result \cite[Theorem 2C*]{G81} cited in the Introduction.
\medskip

It remains now to prove Proposition \ref{propoprim}.  This will be an easy consequence of results in Sections 4 and 5. 

\section{A technical lemma}

\begin{lemma} \label{primd}
Let $(G,\Omega)$ be o-$2$ transitive and $g,h\in G$ with $\suppo(h)\cap \suppo(h^g)=\emptyset$ and $h\neq 1$.
Then there are elements $f,k\in G$ such that $$[h^{-1},h^f][h^{-g},h^{gk}]\neq [h^{-g},h^{gk}][h^{-1},h^f].$$
\end{lemma}

\begin{proof} 
Since $\suppo(h)\cap \suppo(h^g)=\emptyset$, after interchanging $h$ and $h^g$ if necessary, 
we may assume that there are supporting intervals $\Delta_1, \Delta_2:=\Delta_1g$ of $h$ and $h^g$, respectively, such that
 $\delta_1<\delta_2$ for all $\delta_i\in \Delta_i$ ($i=1,2$). 
 Without loss of generality, $\delta_1 h>\delta_1$ for all $\delta_1\in \Delta_1$ (and so $\delta_2 h^g>\delta_2$ for all $\delta_2\in \Delta_2$).
Let $\gamma,\delta, \lambda, \mu\in \Delta_2$ with
$$\gamma<\gamma h^g<\mu h^{-g}<\delta<\lambda<\mu<\delta h^g<\lambda h^g.$$
The six elements 
$$\gamma,\;\gamma h^g,\; \mu h^{-g},\;\delta,\;\mu,\;\lambda h^g \eqno(1) $$
constitute a strictly increasing sequence in $\Delta_1$.  Choose $\xi_{-1},\xi_0\in \Delta_1$ with $\xi_{-1}<\xi_0$, and elements $\xi_1,\xi_2\in \Delta_2$ with  
$$\xi_0<\xi_1<\xi_1 h^g<\xi_2<\xi_2 h^g.$$   Then the six elements
$$\xi_{-1},\;\xi_0,\;\xi_1,\;\xi_1 h^g,\;\xi_2,\;\xi_2 h^g \eqno(2)$$ 
constitute a strictly increasing sequence in $\Omega$.  Using o-$6$-transitivity we can find an element $k$ of $G$ that maps the $n$th element of sequence ($2$) to the $n$th element of sequence ($1$) for each $n$.  
Since $\suppo(h)\cap \suppo(h^g)=\emptyset$ and $\xi_{-1}\in \Delta_1\subseteq \suppo(h)$ we have $\gamma h^{gk}=\gamma k^{-1}h^gk=\xi_{-1}h^gk=\xi_{-1}k=\gamma$.  This and other similar easy calculations show that
$$
\gamma h^{gk}=\gamma,\quad (\gamma h^g)h^{gk}=\gamma h^g,\quad (\mu h^{-g})h^{gk}=\delta,\quad  \mu h^{gk}=\lambda h^g.$$

Now choose $\alpha\in \Delta_1\subseteq \suppo(h)$ and $\beta\in (\alpha h^{-1},\alpha),$ and choose 
$\zeta_1,\dots,\zeta_4\in \suppo(h)$ such that the eight elements 
$$\zeta_4,\;\zeta_3,\;\zeta_2,\;\zeta_1,\;\zeta_4h,\;\zeta_3h,\;\zeta_2h,\;\;\zeta_1h$$
form a strictly increasing sequence. 
Since $\suppo(h)<\suppo(h^g)$, the eight elements 
$$\beta h^{-3},\,\;\alpha  h^{-3},\,\;\beta h^{-2},\,\;\alpha  h^{-2},\,\;\gamma h^{-g},\,\;\gamma ,\,\;\delta,\,\;\lambda$$
also form a strictly increasing sequence, and we can find an element $f\in G$ that maps the $n$th term of the former of these two sequences to the $n$th term of the latter for each $n$.
A routine calculation now shows that  
$$\alpha h^{-2}h^{f}=\lambda, \quad \beta h^{-2}h^{f}=\delta,\quad \alpha h^{-3}h^{f}=\gamma, \;\;\;\hbox{and}\;\;\; \beta h^{-3}h^{f}=\gamma h^{-g}.$$

Let
$$ w_1:=[h^{-1},h^f],\quad \hbox{and} \quad
w_2:=[h^{-g},h^{gk}].$$ 
Further simple calculations show that $$\lambda w_1=\gamma\quad \hbox{and}\quad \lambda w_2=\delta.
$$
Moreover, $$\gamma w_2= \gamma \quad \hbox{and}\quad  \delta w_1=\beta h^{-3}h^f=
\gamma h^{-g} .$$  Hence
$$\lambda w_1w_2=\gamma\neq \gamma h^{-g}=\delta w_1 = \lambda w_2 w_1.$$ 
\vskip-0.22in

\end{proof}

\section{Centralizers: the non-minimal case}
Throughout this section and the next, we assume that $(G,\Omega)$ is a fully depressible transitive $\ell$-permutation group, and write $T$, $\fK$
for its root system and spine.  

For each $h\in G$, define $X_h$, $W_h$ as in Section 2.
For each $\Delta\in T$, let $$Q_\Delta=\{ h\in \rst(\Delta)\mid (\exists \alpha\in \Om)(V(\alpha h,\alpha)=\kappa(\Delta) \}.$$ 
As $(G,\Om)$ is transitive and fully depressible, we have $Q_\Delta \neq \emptyset$. 
Since $(\rst(\Delta))^g=\rst(\Delta g)$ commutes with $\rst(\Delta)$ for $g\notin\St(\Delta)$,  we also have
$$X_h\subseteq \rst(\Delta)\quad\hbox{and} \quad W_h\subseteq \rst(\Delta)\qquad \hbox{for all }\ \Delta\in T\hbox{ and }h\in Q_\Delta.$$

We will use the following observation. 

\begin{remark}\label{trivial}  {\rm Let $(\Lambda,\leq)$ be a totally ordered set and $\frak S$ be a {\em finite} set of pairwise disjoint convex subsets of $\Lambda$. If $f\in \Aut(\Lambda,\leq)$ and $\frak S f=\frak S$, then $Sf=S$ for all $S\in \frak S$.}  \end{remark}

For the rest of this section we assume that {\em $\fK$ has no minimal element}. 
\begin{lemma} \label{40}
Let $\Delta\in T$ and $h\in Q_\Delta$.  
\begin{enumerate}
\item[\rm(a)]
Let  $\Delta'\in T$ with $\Delta'\subset \Delta$ and $\Delta' h\neq \Delta'$, and let $g\in \rst(\Delta')$ with $g\neq 1$. 
\begin{enumerate} 
\item[\rm (i)] 
Then $[h^{-1},h^g]\neq 1$. In particular, $X_h\neq 1$.
\item[\rm (ii)] 
If $f\in G$ and $[[h^{-1},h^g],f]=1$, then $\Delta' f = \Delta'$. In particular,
if $f\in\C_G(X_h)$ then $\Delta'f=\Delta'$.\end{enumerate}
\item[\rm (b)] If $\beta\in \suppo(h)$ and $f\in G$, then either $\beta f= \beta$ or 
$[[h^{-1},h^g],f]\neq1$ for some $g\in\rst(\Delta)$.
\end{enumerate}
\end{lemma}

\begin{proof}  (a) The elements $g^{h^{-1}},g, g^{h}$ have disjoint supports contained in $ \Delta'h^{-1},\Delta'$ and $\Delta'h$ respectively, and so the restrictions of $[h^{-1},h^g]=g^{-h^{-1}}gg^{-h}g$ to these three sets are
conjugates of $g^{-1}$ and $g^2$ and are non-trivial.  Assertion (i) follows. An arbitrary conjugate
$[h^{-1},h^g]^f$ has non-trivial restrictions to the distinct o-blocks $ \Delta'h^{-1}f,\Delta'f$ and $\Delta'hf$, and so if the hypothesis of (ii) holds then
$\{ \Delta'h^{-1},\Delta',\Delta'h\}f=\{\Delta'h^{-1},\Delta',\Delta'h\}$. Thus $f$ must map each of $\Delta'h^{-1},\Delta',\Delta'h$ to itself, by Remark \ref{trivial}.
 
(b) Suppose that $\beta f\neq \beta$. By Proposition \ref{ccc}, one of the convex congruences $V(\beta, \beta h), V(\beta,\beta f)$ contains the other. Let $\Delta'\in T$ be a non-singleton o-block that is strictly contained in the o-block containing $\beta$ for each of these congruences.  Then $\Delta'\subset \Delta$ and $\Delta' f\neq \Delta'$.
Let $g\in Q_{\Delta'}$; then $g\in\rst(\Delta')$ and from (a)(ii) we have $[[h^{-1},h^g],f]\neq1$.
\end{proof}  

\begin{lemma} \label{centd}
Let $\Delta\in T$ and $h\in Q_{\Delta}$.  
\begin{enumerate}
\item[\rm(a)]  $\C_G(X_h)$ contains the pointwise stabilizer of $\Delta$ and is contained in the pointwise stabilizer of $\suppo(h)$.
\item[\rm(b)] $$ W_{h} = \bigcup\{ X_{h^g}\mid g\in \St(\Delta)\}.$$ 
\item[\rm(c)]
$\C_G(W_h)$ is the pointwise stabilizer of $\Delta$.
\end{enumerate}\end{lemma} 

\begin{proof} (a) The first inequality holds since $X_{h}$ moves only points in
$\Delta$. 

Let $f\in \C_{G}(X_{h})$.  Since $X_{h} \subseteq \rst(\Delta)$ we have
$X_h=X_h^f\subseteq \rst(\Delta)\cap \rst(\Delta f)$, and since
$\{\Delta g\mid g\in G\}$ partitions $\Omega$ and $X_h\neq1$ we have 
$\Delta f=\Delta$.  Thus $f\in \St(\Delta)$.
Let $\beta\in \suppo(h)$.  
Then $\beta f=\beta$ by Lemma \ref{40}(b) since $f\in \C_{G}(X_{h})$.

(b) Let $g\in G$.  
If $\Delta g \neq \Delta$, then  $\rst(\Delta)\cap \rst(\Delta g)=1$ and the elements of $X_h$ and $X_{h^g}$ have disjoint support. Thus $[X_h,X_{h^g}]=1$. Hence
$$W_{h} = \bigcup\{ X_{h^g}\mid g\in \St(\Delta), [X_h,X_{h^g}]\neq 1\}.$$ 

Now let $g\in \St(\Delta)$ and $k=h^g$. So $k\in Q_{\Delta}$.  

First suppose that there is some
$\Gamma\in \pi(\Delta)$ with $\Gamma h\neq \Gamma$ and
$\Gamma k\neq \Gamma$. 
Choose $\beta\in \Gamma$ and $\Delta'\in T$ with $\beta\in \Delta'\subset \Gamma$. 
We claim that there is an element $y\in Q_{\Delta'}$
with $\beta y\neq \beta$.  Choose $\beta'\in \Delta'$ such that $\beta$, $\beta'$ belong to different $\pi(\Delta')$-classes.  By transitivity there
is an element $y'\in G$ with $\beta y'=\beta'$. Evidently $\Delta'y'=\Delta'$, and the element $y:=\dep(y',\Delta')$ has the required properties.

Choose $\Delta''\in T$ with $\beta\in \Delta''\subset \Delta'$ and $\Delta''\neq \Delta''y$.  Arguing as above we can find  $x\in Q_{\Delta''}$ with $\beta\in \suppo(x)$.
Write $a:=[k^{-1},k^{x}]=x^{-k^{-1}}xx^{-k}x$.  Thus 
$$\suppo(a) \subset (\Delta'')k^{-1}\cup \Delta'' \cup \Delta'' k \subset (\Delta')k^{-1}\cup \Delta' \cup \Delta' k,$$ 
and the unions above are disjoint unions since $\Delta'\subset \Gamma$ and $\Gamma k\neq\Gamma$.  On the classes 
$\Delta'' k^{-1}, \Delta'', \Delta'' k$ the element $a$ agrees with
$x^{-k^{-1}}, x^2, x^{-k}$ respectively, none of which is the identity since $x\neq 1$.
Therefore
$$\delta' a=\delta' x^2\quad \hbox{for all } \delta'\in \Delta'.$$  
Since $[h^{-1},h^y]=y^{-h^{-1}}yy^{-h}y$ we also have
$$\suppo([h^{-1},h^y]) \subset \Delta'h^{-1}\cup \Delta' \cup \Delta' h,\quad\hbox{and}\quad \delta' [h^{-1},h^y]=\delta' y^2\quad\hbox{for all }\delta'\in \Delta'. $$
Since $\suppo(x)\subseteq \Delta''$ which is disjoint from $\Delta'' y^2$, for any $\delta''\in \suppo(x)$ we have
$$\delta'' [h^{-1},h^y]a =\delta'' y^2x^2=\delta'' y^2\neq \delta'' x^2y^2=\delta'' a[h^{-1},h^y].$$
But $[h^{-1},h^y]\in X_h$ and $a\in X_{k}$.
Hence $[X_h,X_{k}]\neq 1$ and $X_{k}\subseteq W_h$.  

Now suppose instead that $\Gamma k=\Gamma$ or $\Gamma h=\Gamma$ for all $\Gamma \in \pi(\Delta)$. Then $(G(\Delta),\pi(\Delta))$ cannot be abelian and so is of type (II) in Lemma \ref{dichot}.    Lemma \ref{primd} gives elements of $X_h$ and $X_k$ whose images in $(G_\Delta,\pi(\Delta))$ fail to commute, and again we conclude that $X_k\subseteq W_h$.

(c)  The pointwise stabilizer of $\Delta$ lies in $\C_G(W_h)$ 
since $W_h\subseteq \rst(\Delta)$.

Let $\delta\in \Delta$ and $\alpha\in \suppo(h)$.
Choose $g\in \rst(\Delta)$ with $\alpha g=\delta$. So $\delta\in \suppo(h^g)$.
By (a), $\C_G(X_{h^g})$ fixes each point of $\suppo(h^g)$ and so fixes $\delta$.  Since $\C_G(W_h)\subseteq \C_G(X_{h^g})$ by (b), we conclude that 
$\C_G(W_h)$ fixes $\delta$.  The assertion follows.  \end{proof}

\begin{prop}\label{2centd}  Let $\Delta\in T$. Then $\C_{G}^2(W_h)=\rst(\Delta)$ for each $h\in Q_{\Delta}$. In particular, $\C^2_G(W_h)$ is independent of the choice of $h\in Q_\Delta{:}$
$$\C_G^2(W_h)=\C_G^2(W_{h'})\quad \hbox{for all}\;\;  h,h'\in Q_{\Delta}.$$
\end{prop}

\begin{proof} 
By Lemma \ref{centd}(c) the subgroups $\rst(\Delta)$ and $\C_G(W_h)$ commute and so $\rst(\Delta)\subseteq\C_G^2(W_h)$. We must prove that $\C_G^2(W_h)\subseteq\rst(\Delta)$.

Suppose that $z\in G$ with $\Delta z\neq \Delta$. So $\Delta z\cap \Delta=\emptyset$.
Since $(G,\Om)$ is fully depressible, $\rst(\Delta)$ contains an element $x\neq1$; then $x^z\in\rst(\Delta z)$ and so $[x,z]\neq1$.  Thus $z\not\in \C_G^2(W_h)$ by Lemma \ref{centd}(c).  It follows that $\C_G^2(W_h)\subseteq \St(\Delta)$.

Hence if 
 $g\in \C^2_G(W_h)$, then the element $g_0:=\dep(g,\Delta)$ of $\rst(\Delta)$ is defined. From above, $\rst(\Delta)\subseteq \C^2_G(W_h)$
and so  $f:=gg_0^{-1}\in \C_G^2(W_h)$. Suppose that $f\neq 1$. Let $\alpha\in \suppo(f)$ and $\Delta'\in T$ be the  o-block of $V(\alpha,\alpha f)$ with $\alpha\in \Delta'$. Let $\Delta''\subset \Delta'$ with $\alpha\in \Delta''$ and $\Delta''f\neq \Delta''$, and let $y\in Q_{\Delta''}$.  Since $\suppo(y)\subseteq \Delta''$ but $\suppo(y^f)\subseteq \Delta'' f$ we have
$[y,f]\neq1$.  However $y\in\C_G(W_h)$ and $f\in\C_G^2(W_h)$, and we have a contradiction. 
Hence $f=1$ and $g=g_0\in\rst(\Delta)$.
\end{proof}

\section{Centralizers: the minimal case}

Again let $(G,\Om)$ be a fully  depressible transitive $\ell$-permutation group with spine $\fK$.
For the case when the spine $\fK$ of $(G,\Omega)$ has a minimal element we need an extra condition.  
We say that a transitive $\ell$-permutation group $(G,\Om)$ is {\em locally abelian} if its spine $\fK$ has a minimal element $K_0$ and the o-primitive $\ell$-permutation group $(G(\Delta),\pi(\Delta))$ is abelian for each o-block
$\Delta$ of $K_0$.  

For the rest of this section we assume that {\em the spine $\fK$ of $G$ has a minimal element and that $(G,\Omega)$ is not locally abelian.}

The results in the previous section can all be recovered under the above hypotheses on $(G,\Om)$.
This follows from the following observation, {in which we write $f\leq g$
for $f\vee g=g$ and $H_+$ for $\{ g\in G\mid g>1\}$:

\begin{remark} \label{analogues}  
 {\rm Let $(H,\Lambda)$ be an o-$2$ transitive $\ell$-permutation group, and
 $\mu_1,\mu_2\in \Lambda$  with
$\mu_1<\mu_2$.
Let $\gamma_1,\gamma_2\in \Lambda$ with $\mu_1<\gamma_1<\gamma_2 <\mu_2$. 
By o-$3$ transitivity, there is $f\in H$ with $\gamma_1 f=\gamma_2$ and $\mu_i f=\mu_i$ for $i=1,2$.
So $\mu_1=\mu_1 f^n<\gamma_1 f^n<\mu_2 f^n=\mu_2$ for all $n\in \Z$.
Let $\Xi$ be the smallest convex subset of $\Lambda$ containing $\{ \gamma_1 f^n\mid n\in \Z\}$.
Then $\Xi f=\Xi$. Let $g:=\dep(f,\Xi)$. Then $g\in \Aut(\Lambda,\leq)_+$ with $\suppo(g)\subseteq(\mu_1,\mu_2)$ and $\gamma_1 g=\gamma_2$. 
} 
\end{remark}

The next three results extend the corresponding results (Lemmata 4.2 and 4.3 and Proposition 4.4) of Section 4.

\begin{lemma} \label{50}
Let $\Delta\in T$ and $h\in Q_\Delta$.  
Let $g\in \rst(\Delta)_+$ and suppose that there is $\alpha\in \suppo(h)$ such that $\suppo(g)\subseteq(\alpha,\alpha h)$ if $\alpha < \alpha h$ or
$\suppo(g)\subseteq(\alpha h,\alpha)$ if $\alpha h< \alpha$. 
\begin{enumerate}   
\item[\rm (a)] \begin{enumerate} \item[\rm(i)]  $[h^{-1},h^g] \neq 1;$ in particular  $X_h\neq 1$.
\item[\rm (ii)] 
If $f\in G$ and $[[h^{-1},h^g],f]=1$, then $\suppo(g^{h^i})f = \suppo(g^{h^i})$ for $i=0,\pm 1$. \end{enumerate}
\item[\rm (b)] If $f\in G$, then either $\beta f= \beta$ for all $\beta\in \suppo(h)$ or 
$[[h^{-1},h^g],f]\neq1$ for some $g\in\rst(\Delta)_+$.
\end{enumerate}
\end{lemma}

\begin{proof}  (a) We assume that $\alpha < \alpha h$, the proof when $\alpha h< \alpha$ being similar.

Write $c:=[h^{-1},h^{g}]=g^{-h^{-1}}gg^{-h}g$.  For $i=0,\pm 1$ we have $\suppo(g^{h^i})\subseteq (\alpha h^i, \alpha h^{i+1})$. These intervals are pairwise disjoint and $\suppo(c)$ lies in their union. Since $\suppo(c)$ may not be convex and the intervals may not be mapped to themselves by $f$, we cannot apply Remark \ref{trivial}. This is where we use that $g>1$. The restriction of $c$ to $(\alpha,\alpha h)$ is $g^2>1$ and  
$c$ is strictly positive only on $\suppo(g)$.
Moreover, if $[c,f]=1$, then $(c\vee 1)^f=c^f\vee 1^f=c\vee 1$, so $f$ must conjugate $g^2=c\vee 1\in G$ to itself,
and $c^{-1} \vee 1$ to itself. 
Since $f$ is order-preserving and $\suppo(g^{h^{-1}}) < \suppo(g) < \suppo(g^h)$ and $\suppo(g^2)=\suppo(g)$,
we must have  $\suppo(g^{h^i})f= \suppo(g^{h^i})$ for $i=\pm 1$.

(b) Suppose that $\beta f\neq \beta$ for some  $\beta\in \suppo(h)$ and that $[h^{-1},h^g]$ and $f$ commute for all $g\in\rst(\Delta)_+$.
Let $K_0$ be the minimal element of $\fK$ and 
$\Lambda$ be the $K_0$ o-block with $\beta\in \Lambda\subseteq \Delta$.  Then 
$(G(\Lambda),\pi(\Lambda))$ is o-primitive and o-$2$ transitive
since $(G,\Om)$ is not locally abelian.  
Now $\beta f$ and $\beta h$ are distinct from $\beta$ and so there is an 
interval $(\beta_1,\beta_2)$ containing $\beta$ and disjoint from $(\beta_1 f,\beta_2 f)$  
and $(\beta_1 h,\beta_2 h)\cup (\beta_1 h^{-1},\beta_2 h^{-1})$.
By Remark \ref{analogues} with $\mu_1=\beta_1$ and $\mu_2=\beta_2$,  there is $g\in G_+$ with $\beta\in \suppo(g) \subseteq 
(\beta_1,\beta_2)$. Thus the sets $\suppo(g^{h^i})$ for $i\in\{0,\pm1\}$ are pairwise disjoint and since  $[h^{-1},h^{g}]=g^{-h^{-1}}gg^{-h}g$ 
we obtain that $\beta [h^{-1},h^g]f=\beta g^2f>\beta f$. Since $[h^{-1},h^g]$, $f$ commute, it follows that  
$$\beta f[h^{-1},h^g]>\beta f.$$
Thus $\beta f\in\suppo ([h^{-1},h^g])\subseteq \suppo(g)\cup\suppo(g^h)\cup\suppo(g^{h^{-1}})$. 
Since $\suppo(g)\subseteq (\beta_1,\beta_2)$ and $\beta f\in(\beta_1f,\beta_2f)$ we have $\beta f\notin\suppo(g)$, whereas if
$\beta f\in \suppo(g^h)$ then $\beta f[h^{-1},h^g]=\beta fg^{-h}<\beta f$, 
 and if $\beta f\in \suppo(g^{h^{-1}})$, then $\beta f[h^{-1},h^g]=\beta fg^{-(h^{-1})}<\beta f$.  A contradiction ensues and the lemma is proved.
\end{proof}

\begin{lemma} \label{5centd}
Let $\Delta\in T$ and $h\in Q_{\Delta}$. 
\begin{enumerate}
\item[\rm(a)] $\C_G(X_h)$ contains the pointwise stabilizer of $\Delta$ and is contained in the pointwise stabilizer of $\suppo(h)$.
\item[\rm(b)] $$ W_{h} = \bigcup\{ X_{h^g}\mid g\in \St(\Delta)\}.$$ 
\item[\rm(c)]
$\C_G(W_h)$ is the pointwise stabilizer of $\Delta$.
\end{enumerate}\end{lemma}

\begin{proof} The proofs of (a) and (c) are identical to those of Lemma \ref{centd}. The same is true for (b) in the case when there are $\Delta'',\Delta'\in T$ with $\Delta''\subset \Delta'\subset \Delta$.
It remains to consider the cases when $\Delta$ is minimal in $T$ or covers a minimal element of $T$.
First assume that $\Delta$ is minimal in $T$. 

 Let $g\in G$.  
If $\Delta g \neq \Delta$, then  $\rst(\Delta)\cap \rst(\Delta g)=1$ and the elements of $X_h$ and $X_{h^g}$ have disjoint support. Thus $[X_h,X_{h^g}]=1$. Hence
$$W_{h} = \bigcup\{ X_{h^g}\mid g\in \St(\Delta), [X_h,X_{h^g}]\neq 1\}.$$ 

Now let $g\in \St(\Delta)$ and $k=h^g$. So $k\in Q_{\Delta}$. 
 
First suppose that there is some
$\delta\in \Delta$ with $\delta \in \suppo(h) \cap \suppo(k)$.
Assume that $\delta h>\delta$ and $\delta k>\delta$, the other three cases being similar. 
Choose $\lambda_1<\delta< \min\{ \lambda_1 h,\lambda_1k\}$ and $\lambda_2\in (\delta,\min\{ \lambda_1 h,\lambda_1k\})$.
So $(\lambda_1,\lambda_2)h=(\lambda_1h,\lambda_2h)\subseteq (\lambda_2,\lambda_2h)$ since $\lambda_2<\lambda_1h$; and $(\lambda_1,\lambda_2)h^{-1}=(\lambda_1h^{-1},\lambda_2h^{-1})\subseteq (\lambda_1h^{-1},\lambda_1)$ since $\lambda_2h^{-1}<\lambda_1hh^{-1}=\lambda_1$.
Similarly $(\lambda_1,\lambda_2)k\subseteq (\lambda_2,\lambda_2k)$ and $(\lambda_1,\lambda_2)k^{-1}\subseteq (\lambda_1k^{-1},\lambda_1)$.

Since $(G,\Om)$ is fully  depressible,  by Remark \ref{analogues} applied for $\mu_1=\lambda_1$ and $\mu_2=\lambda_2$ there is $y\in \rst(\Delta)_+$ with $\delta\in \suppo(y)$ and $\suppo(y)\subseteq (\lambda_1,\lambda_2)$.  Define $b:=[h^{-1},h^y]=y^{{-h^{-1}}}yy^{-h}y$.
By the previous paragraph, $\suppo(y^{h^i})\cap (\lambda_1,\lambda_2)=\emptyset$ for $i\in\{-1,1\}$ and so $\lambda b^{-1}=\lambda y^{-2}$ for all $\lambda\in (\lambda_1,\lambda_2)$.  

By Remark \ref{analogues}  (this time applied for
$\mu_1=\delta$ and $\mu_2=\delta y$), we can find  $x\in G_+$  and $\beta\in \suppo(y)$ with $\beta\in \suppo(x)\subseteq (\delta,\delta y)$.  Define
$a:=[k^{-1},k^{x}]=x^{-k^{-1}}xx^{-k}x$.  From above, the sets $\suppo(x^{k^{i}})\subseteq (\delta,\delta y)k^{i}\subseteq (\lambda_1,\lambda_2)k^i$ for $i\in\{0,\pm1\}$ are disjoint
and so $(\lambda_1,\lambda_2)\cap\suppo(a)=(\lambda_1,\lambda_2)\cap\suppo(x)$. Moreover
$\lambda a=\lambda x^2\in (\lambda_1,\lambda_2)$ for all $\lambda\in (\lambda_1,\lambda_2)$ and
$\beta a=\beta x^2\neq \beta$.

\quad Since $\beta\in (\delta,\delta y)$ we have $\beta y^{-2}\in (\delta y^{-2},\delta y^{-1})\subseteq (\lambda_1,\lambda_2)$. 
But $(\delta y^{-2},\delta y^{-1})\cap \suppo(a)=(\delta y^{-2},\delta y^{-1})\cap \suppo(x)=\emptyset$; therefore $a$ fixes $\beta y^{-2}=\beta b^{-1}$ and $a^b$ fixes $\beta$.
Thus $\beta a^b\neq \beta a$ and so $a^b\neq a$.  However, $a\in X_{k}$ and $b\in X_h$.   Hence $[X_h,X_{k}]\neq 1$ and $X_{k}\subseteq W_h$.

Now suppose that each element of $\Delta$ is fixed by $h$ or $k$. Since the minimal o-primitive component is of type (II),   Lemma \ref{primd} applies and provides elements of $X_h$ and $X_k$ whose images in the minimal o-primitive component fail to commute.
This completes the proof of (b) in the case when $\Delta$ is minimal in $T$.  An easy adaptation gives the proof in the case when $\Delta$ covers a minimal element of $T$. \end{proof}

\begin{prop}\label{2centm}
For every $\Delta\in T$ and $h\in Q_{\Delta}$, 
$\C_{G}^2(W_h)=\rst(\Delta).$    Thus
if $\beta\in \Delta$ and $h'\in \rst(\Delta)$  with $V(\beta,\beta h')=\kappa(\Delta)$, then $$\C_G^2(W_h)=\C_G^2(W_{h'}).$$ In particular, 
$$\C_G^2(W_h)=\C_G^2(W_{h'})\quad \hbox{for all}\;\;  h,h'\in Q_{\Delta}.$$
\end{prop}

\begin{proof} 
By Lemma \ref{5centd}(c) we have $\rst(\Delta)\subseteq \C^2_G(W_h)$, and the argument at the corresponding point in the proof of Lemma \ref{2centd} shows that
$\C_G^2(W_h)\subseteq\St (\Delta)$.

Let $g\in \C^2_G(W_h)$ and $g_0:=\dep(g,\Delta)\in \rst(\Delta)\subseteq C_G^2(W_h)$.
Thus $f:=gg_0^{-1}\in \C_G^2(W_h)$ and $\suppo(f)\cap \Delta=\emptyset$. 
If $f\neq 1$, let $\alpha\in \suppo(f)$ and $\Delta'$ be the minimal o-block in $T$ with $\alpha\in \Delta'$. We assume that $\alpha f>\alpha$, the other case being similar. 
Since $(G(\Delta'),\pi(\Delta'))$ is o-primitive (and so o-$2$ transitive) and since $(G,\Omega)$ is fully depressible, Remark \ref{analogues} (with $\beta_1=\alpha, \beta_2= \alpha f)$ yields a non-trivial element $y\in \rst(\Delta')$ with $\suppo(y)\subseteq (\alpha,\alpha f)\subseteq \suppo(f)$.  Since $\suppo(f)\cap \Delta=\emptyset$, we have $\delta y=\delta$ for all $\delta\in \Delta$ and thus $y\in \C_G(W_h)$ by Lemma \ref{5centd}(c).
But $y^f\neq y$ since $\suppo(y^f)\subseteq(\alpha f,\alpha f^2)$.
This contradicts that $y\in \C_G(W_h)$ and $f\in \C_G^2(W_h)$. 
Hence $f=1$ and every element of $\C_{G}^2(W_h)$ lies in $\rst(\Delta)$.
\end{proof}

\section{Proof of Proposition \ref{propoprim}}

\begin{proof}
Let $(G,\Om)$ be a transitive fully  depressible $\ell$-permutation group. If $(G,\Om)$ is o-primitive 
then $T=\{\Omega\}$;
if $(G,\Om)$ is also abelian, then $X_g=\{ 1\}$ for all $g\in G\setminus \{1 \}$ and so $\C^2_G(W_g)=\C_G(G)=G$, whereas if $(G,\Om)$ is non-abelian and $g\in G\setminus \{ 1\}$,
then $g\in Q_\Omega$ and $\C^2_G(W_g)=G$ by Proposition \ref{2centm}.

Now suppose that  $(G,\Om)$ is not o-primitive and choose $\Delta\in T$ with $\Delta\neq \Om$. By full depressibility there is an element $g\in Q_\Delta$, and by transitivity $\Delta f\cap \Delta=\emptyset$ for some $f\in G$. Since $g^f\in Q_{\Delta f}$ we have $[g,f]\neq1$, and $G$ is not abelian.
If $(G,\Om)$ is not locally abelian, then $\C^2_G(W_g)=\rst(\Delta)$ is disjoint from $\rst(\Delta f)=\C^2_G(W_{f^{-1}gf})$ by
Propositions \ref{2centd} and \ref{2centm};
so $\C^2_G(W_g)\neq G$.
If instead  $(G,\Om)$ is locally abelian, let $\Delta\in T$ be minimal and $g\in Q_\Delta$. Then $\suppo(g)=\Delta$ and for any $f\in G$ either $\Delta  f\cap \Delta=\emptyset$ or $\suppo(g^f)=\Delta$. In each case, $[g^{-1},g^f]=1$. Thus $W_g=\emptyset$ and $\C^2_G(W_g)=\C_G(G)\neq G$, and the proposition is established. 
\end{proof}

\medskip

The proof of Proposition \ref{propoprim} completes the proof of Theorems \ref{reals} and \ref{GHQ2}.

\bigskip
  
{\noindent {\bf Acknowledgment.} {\rm This research was begun when the second author was the Leibniz Professor at the University of Leipzig. The authors are most grateful to the Research Academy, Leipzig and the Leibniz Program of the University of Leipzig for funding a visit by the first author that made this research possible.}

\bigskip

\noindent A. M. W. Glass,

\noindent QUEENS' COLLEGE, CAMBRIDGE CB3 9ET,
U.K.\medskip

E-mail: amwg@dpmms.cam.ac.uk \bigskip

\noindent John S. Wilson,

\noindent MATHEMATICAL INSTITUTE, ANDREW WILES BUILDING, 

\noindent WOODSTOCK ROAD, OXFORD OX2 6GG, U.K.

E-mail: John.Wilson@maths.ox.ac.uk

\end{document}